\documentclass[12pt,a4paper]{article}
\usepackage[T1]{fontenc}
\usepackage{mathtools}
\usepackage{amssymb}
\usepackage{amsthm}
\usepackage{amsmath}
\usepackage{xcolor}
\usepackage[numbers]{natbib}
\usepackage{hyperref}
\hypersetup{
	colorlinks=true,
	citecolor=blue,
	linkcolor=blue,  
	urlcolor=blue    
}
\usepackage{enumitem}
\usepackage{epstopdf}
\usepackage{subcaption}
\usepackage{authblk}
\author[1]{Yashaswini H L}
\author[1]{Vinay Madhusudanan}
\author[1]{Kavitha Koppula}
\author[1]{Kedukodi Babushri Srinivas}
\author[1]{Kuncham Syam Prasad\textsuperscript{*}}
\affil[1]{Department of Mathematics, Manipal Institute of Technology, Manipal Academy of Higher Education, Karnataka, India}
\date{*Corresponding author email: syamprasad.k@manipal.edu}
\theoremstyle{plain}
\newtheorem{theorem}{Theorem}[section]
\newtheorem{lemma}[theorem]{Lemma}
\newtheorem{corollary}[theorem]{Corollary}
\newtheorem{proposition}[theorem]{Proposition}

\theoremstyle{definition}
\newtheorem{definition}[theorem]{Definition}
\newtheorem{example}[theorem]{Example}

\newtheorem{remark}[theorem]{Remark}

\DeclareMathOperator{\id}{id}
\DeclareMathOperator{\im}{im}
\DeclareMathOperator{\gr}{gr}
\DeclareMathOperator{\CIG}{CIG}
\DeclareMathOperator{\SIG}{SIG}

\title{On Complement and Supplement Ideals of Nearrings}

\begin{document}

	\maketitle
	\begin{abstract}
		In this article we study complement ideals, and the dual concept of supplement ideals, in nearrings, both of which are generalizations of the concept of complement in a bounded modular lattice. We prove fundamental properties of complements and supplements in arbitrary nearrings. We then establish Galois connections between the ideal lattices of a nearring and of its matrix nearrings, yielding one-to-one correspondences between their respective complement and supplement ideals. We also define graphs associated with complement and supplement ideals of nearrings and study some of their combinatorial properties such as girth and clique number.
	\end{abstract}
	
	\subsection*{Mathematical Subject Classification}
		16Y30; 16S50; 05C25
	
	\subsection*{Keywords}
		Nearring; matrix nearring; complement ideal; supplement ideal; graph
	
	\section{Introduction} \label{sec:Intro}
	
	A nearring is a generalization of a ring. Rings can be considered as arising from group endomorphisms, and therefore as consisting of linear maps, in some sense, whereas nearrings naturally arise from the study of general maps on groups, and thus are non-linear. As a result, the theory of nearrings differs considerably from that of rings. The crucial difference between rings and nearrings is that the additive group of a nearring can be non-Abelian, and distributivity can be one-sided. Two related concepts are those of $N$-groups, which are nearring analogues of modules over a ring, and matrix nearrings. Meldrum and Van der Walt \cite{meldrum1986matrix} defined the notion of a matrix nearring over a nearring $N$, denoted by $M_n(N)$, as a particular subnearring of the nearring $M(N^n)$, the set of all maps from $N^n$ to $N^n$. Booth and Groenewald \cite{booth1991primeness} established a one-to-one correspondence between the set of equiprime ideals of $N$ and that of $M_n(N)$. Every ideal of $N$ can be mapped to ideals in $M_n(N)$ in two well-known ways, and all the ideals of $M_n(N)$ lying between these two ideals map back to the original ideal of $N$. These ideals form an interval in the ideal lattice $\mathfrak L(M_n(N))$. Meldrum and Meyer \cite{meldrum1996intermediate} have shown that any ideal lattice can occur as such an interval. Salvankar et al. \cite{salvankar2024essential} obtained one-to-one correspondences between essential ideals of a nearring and those of its matrix nearrings, and between superfluous ideals of a nearring and those of its matrix nearrings. For recent developments in matrix nearrings, we refer to \cite{meldrum1986matrix, meldrum1996intermediate, veldsman1992special, meyer1994chains}. The concepts of complement ideals and supplement $N$-subgroups were studied by Bhavanari and Reddy \cite{bhavanari1988, satyanarayana1991finite}.
	
	The study of algebraic structures using graphs began with the introduction of certain graphs associated with groups, now known as Cayley graphs, by Cayley \cite{cayley1878desiderata}. Later, several graphs were defined on groups, such as commuting graphs, and on rings, notably zero divisor graphs. In the same vein, Bhavanari et al., in \cite{Bhavanari26042010}, defined and studied the graph of a nearring with respect to an ideal, and investigated the structural properties of these graphs for different types of prime ideals such as 3-prime, c-prime, and equiprime ideals. A different approach uses graphs defined taking the substructures of an algebraic structure -- such as ideals of a ring or a nearring -- as vertices. Salvankar et al., \cite{rajani2023superfluous}, studied such graphs associated with generalized superfluous ideals of nearrings.  
	
	This paper develops the theory of complement and supplement ideals in nearrings, and is organized as follows. In Section~\ref{sec:ComplementsAndSupplements} we obtain preliminary results on complement and supplement ideals of nearrings. In Section~\ref{sec:MatrixNearrings} of this article, we observe Galois connections between $\mathfrak L(N)$ and $\mathfrak L(M_n(N))$, and also establish one-to-one correspondences between the complement and supplement ideals of a nearring and of its matrix nearrings. In Section~\ref{sec:Graphs} we define graphs associated with complement and supplement ideals and obtain some results related to these. Also, we study combinatorial properties of graphs such as clique number and girth.
	
	In this paper we consider right nearrings, i.e., nearrings in which right distributivity holds. A normal subgroup $I$ of $(N, +)$ is an \emph{ideal} if $IN \subseteq I$ and $n(n' + i) - nn' \in I$ for $n, n' \in N$ and $i \in I$. We refer to \cite{pilz, satyanarayana2013near} for the basic theory of nearrings. 
	
	\section{Complement and Supplement Ideals of Nearrings}\label{sec:ComplementsAndSupplements}
	
	In this section, we prove some general results on complement ideals and supplement ideals of nearrings, and in particular, discuss ideals in Malone trivial nearrings.
	
	\begin{definition}[\cite{satyanarayana2013near}]
		An ideal $I$ of $N$ is a \emph{complement} of an ideal $K$ of $N$ if it is maximal with respect to the property $K \cap I = \{0\}$.
	\end{definition}
	
	\begin{definition}
		An ideal $I$ of $N$ is a \emph{supplement} of an ideal $K$ of $N$ if it is minimal with respect to the property $K + I = N$.
	\end{definition}
	
	\begin{example}
		Let $N = (\mathbb{Z}_4\times \mathbb{Z}_2, +, \star)$. The addition is component-wise and multiplication is defined as given in the table below.
		
		\begin{center}
			\begin{tabular}{|c|c|c|c|c|c|c|c|c|}
				\hline 
				$\star$ & $0$ & $1$ & $2$ & $3$ & $4$ & $5$ & $6$ & $7$ \\
				\hline
				$0$ &  $0$ &  $0$ &  $0$ &  $0$ &  $0$ &  $0$ &  $0$ &  $0$ \\
				\hline
				$1$ &  $0$ &  $0$ &  $0$ &  $2$ &  $0$ &  $2$ &  $0$ &  $2$ \\
				\hline
				$2$ &  $0$ &  $0$ &  $0$ &  $0$ &  $0$ &  $0$ &  $0$ &  $0$ \\
				\hline
				$3$ &  $0$ &  $0$ &  $0$ &  $2$ &  $0$ &  $2$ &  $0$ &  $2$ \\
				\hline
				$4$ &  $0$ &  $0$ &  $0$ &  $0$ &  $2$ &  $2$ &  $2$ &  $2$ \\
				\hline
				$5$ &  $0$ &  $0$ &  $0$ &  $2$ &  $2$ &  $0$ &  $2$ &  $0$ \\
				\hline
				$6$ &   $0$ &  $0$ &  $0$ &  $0$ &  $2$ &  $2$ &  $2$ &  $2$ \\
				\hline
				$7$ &  $0$ &  $0$ &  $0$ &  $2$ &  $2$ &  $0$ &  $2$ &  $0$ \\
				\hline
			\end{tabular} 
		\end{center}
		
		The ideals of $N$ are $I_1 = \{0\}$, $I_2 = \{0, 2\}$, $I_3 = \{0, 2, 5, 7\}$, $I_4 =\{0, 1, 2, 3\}$, $I_5 = \{0, 2, 4, 6\}$, and $ N $. Observe that $I_3$ is a supplement of $I_4$.
	\end{example}
	
	Observe that the definitions of complement and supplement are dual to each other. As a result, most of the properties proved below also have dual versions, and we shall assume henceforth that this is understood. Wherever the dual result does not hold, we explicitly state so.
	
	A nearring $N$ is \emph{indecomposable} if it cannot be expressed as a direct sum of two proper ideals.
	
	\begin{lemma}\label{lem:indecomposable}
		If $N$ is indecomposable, then no maximal ideal of $N$ is a complement and no minimal ideal of $N$ is a supplement.
	\end{lemma}
	
	\begin{proof}
		Suppose that $M$ is a maximal ideal of $N$ which is a complement. Then there exists an ideal $I \neq \{0\}$ such that $M \cap I = \{0\}$. Hence $I + M = N$ which is a contradiction. Therefore any maximal ideal of $N$ cannot be a complement. Similarly, we can prove that no minimal ideal of $N$ can be a supplement. 
	\end{proof}
	
	\noindent The converse of Lemma~\ref{lem:indecomposable} is not true, as shown in the following example.
	
	\begin{example}
		Let $N$ be the zero ring defined on the group $\mathbb{Z}_4\times \mathbb{Z}_4$. The ideal lattice of $N$ is as shown in Figure~\ref{fig:L(Z4xZ4)}.
		
		\begin{figure}[h]
			\centering
			\includegraphics[width=0.35\textwidth]{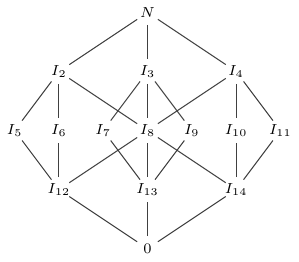}
			\caption{\label{fig:L(Z4xZ4)} Ideal lattice of $N$}
		\end{figure}
		
		We can observe that the maximal ideals $I_2$, $I_3$, and $I_4$ of $N$ are not complement ideals and the minimal ideals $I_{12}$, $I_{13}$, and $I_{14}$ are not supplements. However, the ideals $I_5$ and $I_9$ are direct summands of $N$.
	\end{example}
	
	An ideal of $N$ is \emph{essential} if it has a non-trivial intersection with every non-trivial ideal of $N$, and it is \emph{superfluous} if its sum with every proper ideal of $N$ is proper.
	
	\begin{lemma}\label{lem:IcapJissuperfluous}
		Let $J$ be a supplement ideal of an ideal $I$ of $N$. Then $I \cap J$ is a superfluous ideal in $N$. 
	\end{lemma}
	
	\begin{proof}
		Suppose that there exists an ideal $K$ such that $(I \cap J) + K = N$. Then $((I\cap J) + K) \cap J = J$. By modularity, $(I \cap J) + (J \cap K) = J$. Now, $N = I + J = I + (I \cap J) + (J \cap K) = I + (J \cap K)$. Since $J$ is a supplement of $I$, $J \cap K = J$. This implies that $J\subseteq K$. Therefore, $K = (I \cap J) + K = N$. Thus $I \cap J$ is a superfluous ideal in $N$. 
	\end{proof}
	
	\begin{remark}\label{rem:SuperfluousSupplement}
		The only superfluous supplement ideal of $N$ is $\{0\}$.
	\end{remark}
	
	\begin{lemma}\label{lem:0andN}
		If one of the following conditions holds, then $\{0\}$ and $N$ are the only supplement ideals of $N$:
		\begin{enumerate}[label = (\roman*)]
			\item $N$ contains a greatest proper ideal.
			\item $N$ has ACCI and contains a unique maximal ideal.
		\end{enumerate}
	\end{lemma}
	
	\begin{proof}
		Suppose that $N$ contains a greatest proper ideal $M$. Then, since every proper ideal of $N$ is contained in $M$, the sum of any two proper ideals is proper, and therefore no proper non-trivial ideal is a supplement.
		
		If $N$ has ACCI and contains a unique maximal ideal $M$, then $M$ is the greatest proper ideal of $N$.
	\end{proof}
	
	\begin{lemma}\label{lem:AtLeast2Supplements}
		If $N$ has DCCI, then $N$ cannot have a unique proper non-trivial supplement. 
	\end{lemma}
	
	\begin{proof}
		Suppose that $I$ is a proper non-trivial supplement in $N$. Then there exists a proper non-trivial ideal $J$ of $N$ such that $I$ is minimal with respect to $I + J = N$. Since $N$ has DCCI, there exists a proper ideal $K \subseteq J$ which is minimal with respect to the property $I + K = N$. Hence there exists at least two proper non-trivial supplements in $N$.  
	\end{proof}
	
	\begin{lemma}\label{lem:UniqueMaximalIdeal}
		Let $N$ have ACCI and DCCI. Then $N$ has a unique maximal ideal if and only if $\{0\}$ and $N$ are the only supplement ideals of $N$.
	\end{lemma}
	
	\begin{proof}
		Suppose that $\{0\}$ and $N$ are the only supplements in $N$. Since $N$ has ACCI, every chain of ideals has a maximal element. Suppose that $M_1$ and $M_2$ are two maximal ideals of $N$. Then $M_1 + M_2 = N$. Consider a chain of ideals $\{I_\alpha\}$, $I_\alpha \subseteq M_2$, such that $I_\alpha + M_1 = N$. Since $N$ has DCCI, there exists a minimal ideal $\mathfrak{m}$ such that $\mathfrak{m} \subseteq M_2$ and $\mathfrak{m} + M_1 = N$, which is a contradiction. Therefore, $N$ has a unique maximal ideal. The converse follows from Lemma~\ref{lem:0andN}.
	\end{proof}
	
	\begin{lemma}\label{lem:DirectImpliesComplementAndSupplement}
		If $I$ is a direct summand of $N$, then $I$ is both a complement and supplement ideal.
	\end{lemma}
	
	\begin{proof}
		Since $I$ is a direct summand, there exists $J$ in $N$ such that $I + J = N$ and $I \cap J = \{0\}$. Suppose that there exists an ideal $K$ such that $K \subseteq I$ and $K + J =N$. Now, $I = I\cap N = I \cap (K + J)$. By modularity, $I = K + (I\cap J) = K$. Therefore $I$ is a supplement. Similarly, we can prove $I$ is a complement.  
	\end{proof}
	
	The converse of Lemma~\ref{lem:DirectImpliesComplementAndSupplement} is not true. For example, consider the zero nearring $N$ on the group $\mathbb{Z}_2^2 \times \mathbb{Z}_8$. This nearring has ideals of order $4$ that have $\mathbb{Z}_4$ as their underlying group, and these ideals cannot be direct summands of $N$. But we can observe that these ideals are both complement and supplement in $N$.
	
	Note that if $N = N_1 \oplus N_2$ is a nearring with a right identity, then the ideals of $N$ are exactly the direct sums $I_1 \oplus I_2$, where $I_1$ is an ideal in $N_1$ and $I_2$ is an ideal in $N_2$.
	
	\begin{lemma}\label{lem:supplementsindirectsum}
		If $N$ is a nearring with a right identity, and $N = N_1 \oplus N_2$ is a direct decomposition of $N$, then every supplement ideal of $N$ is of the form $I_1 \oplus I_2$ where $I_\alpha$ is a supplement ideal in $N_\alpha$, $\alpha = 1, 2$.
	\end{lemma}
	\begin{proof}
		First, suppose that $I_\alpha$ is a supplement of $J_\alpha$ in $N_\alpha$, $\alpha = 1, 2$. Let $I = I_1 \oplus I_2$ and $J = J_1 \oplus J_2$. Then $I + J = (I_1 + J_1) \oplus (I_2 + J_2) = N$. To see that $I$ is minimal with respect to this property, suppose that there exists an ideal $K$ in $N$ such that $K \subseteq I$ and $K + J = N$. Then $K = K_1 \oplus K_2$, where $I_\alpha \supseteq K_\alpha \unlhd N_\alpha$, and $K_\alpha + J_\alpha = N_\alpha$, $\alpha = 1, 2$. Since $I_\alpha$ is a supplement of $J_\alpha$ in $N_\alpha$, it follows that $K_\alpha = I_\alpha$. Thus, $K = I$, as required.
		
		For the converse, suppose that $I = I_1 \oplus I_2$ is a supplement of $J = J_1 \oplus J_2$ in $N$. Then $I + J = N$ implies that $I_\alpha + J_\alpha = N_\alpha$, $\alpha = 1, 2$. To see that $I_\alpha$ is minimal with respect to this property, suppose that $I_\alpha \supseteq K_\alpha \unlhd N_\alpha$, with $K_\alpha + J_\alpha = N_\alpha$, $\alpha = 1, 2$. Then $(K_1 \oplus K_2) + (J_1 \oplus J_2) = N$ and $I_1 \oplus I_2 \supseteq K_1 \oplus K_2$, which implies that $K_1 \oplus K_2 = I_1 \oplus I_2$, i.e., $K_\alpha = I_\alpha$, $\alpha = 1, 2$. 
	\end{proof}
	
	Now we prove some properties of ideals in Malone trivial nearrings that will be used in Section~\ref{sec:Graphs}.
	
	\begin{definition}[\cite{veldsman1992equiprime}]\label{def:Malone}
		Let $G$ be a non-trivial group and $S \subseteq G \setminus \{0\}$. Let $N$ be the nearring defined on $G$ with multiplication given by
		
		\begin{equation*}
			ab = 
			\begin{cases}
				a, & \text{if}~ b \in S \\
				0, & \text{if}~ b \notin S.
			\end{cases}
		\end{equation*}	
		The nearring $N$ is a \emph{Malone trivial nearring} on $G$.
	\end{definition}
	
	Throughout the remainder of this section, let $N$ denote a Malone trivial nearring and $S \subseteq N$ be as given in Definition~\ref{def:Malone}.
	
	The following result characterizes an ideal of a Malone trivial nearring as a normal subgroup of the additive subgroup such that $S$ can be expressed as a union of one or more of its cosets other than itself.
	
	\begin{lemma}[{\cite[Proposition 1.6.13]{ferrero2002nearrings}}]\label{lem:MaloneIdeal}
		A proper normal subgroup $I$ of $(N, +)$ is an ideal of $N$ if and only if $S = \bigcup\limits_{s \in S} (s + I)$ and $S\cap I = \varnothing$.
	\end{lemma}
	
	\begin{lemma}\label{lem:N!=I+J}
		$(N, +, \cdot)$ is indecomposable and it does not have any proper non-trivial supplement.
	\end{lemma}
	
	\begin{proof} 
		Let $I$ and $J$ be proper ideals of $N$, and suppose that $I + J = N$. By Lemma~\ref{lem:MaloneIdeal}, a normal subgroup $I$ of $(N, +)$ is an ideal of $N$ if and only if $S = \bigcup\limits_{s \in S} (s + I)$. Since $I + J = N$, for all $n \in N$, there exists $i \in I$ such that $n + J = i + J$. This implies that $i \in S$, which is a contradiction. Therefore $N \neq I + J$. 
	\end{proof}
	
	\begin{lemma}\label{lem:maximalideal} 
		Every proper ideal of $(N, +, \cdot)$ is contained in a maximal ideal.
	\end{lemma}
	
	\begin{proof} 
		Let $I$ be a proper ideal of $N$. Take $\mathcal P$ to be the set of all proper ideals of $N$ containing $I$. Let $\mathcal L$ be any totally ordered subset of $\mathcal P$. Let $M = \bigcup\limits_{P \in \mathcal L} P$. Then $M$ is an upper bound of $\mathcal L$. Since $M$ contains every $P \in \mathcal L$ and any $P \cap S = \varnothing$, it follows that $M$ is a proper ideal and is in $\mathcal P$. Therefore, by Zorn's lemma, $\mathcal P$ has a maximal element. Hence every proper ideal of $N$ is contained in a maximal ideal. 
	\end{proof}
	
	The following result, which is immediate from Lemmas~\ref{lem:N!=I+J} and \ref{lem:maximalideal}, shows that $N$ has a greatest proper ideal.
	
	\begin{theorem}\label{lem:uniqueness}
		$(N, +, \cdot)$ has a unique maximal ideal $M$, and every proper ideal of $N$ is contained in $M$.
	\end{theorem}
	
	\section{Complement and Supplement Ideals of Matrix Nearrings} \label{sec:MatrixNearrings}
	
	In this section, we prove that $((~)_ *,(~)^*)$ and $((~)^+, (~)_*)$ are Galois connections between the ideal lattice of a nearring $N$ and the ideal lattice of the matrix nearring $M_n(N)$, and we give one-to-one correspondences between complement and supplement ideals of $N$ and those of $M_n(N)$.
	
	Throughout this section, let $N$ denote a right nearring with $1$, and $N^n$ the direct sum of $n$ copies of $(N, +)$. The elements of $N^n$ are column vectors, and are written as $(a_1, \cdots , a_n)$. The symbols $\iota_i$ and $\pi_j$, respectively, denote the $i$\textsuperscript{th} coordinate injective and $j$\textsuperscript{th} coordinate projective maps. For an element $a \in N$ , $\iota_i(a) = (0, \ldots, \underbrace{a}_{i\textsuperscript{th}}, \ldots, 0)$ and $\pi_j (a_1, \ldots , a_n) = a_j$, for any $(a_1, \ldots , a_n)\in N^n$. The nearring of $n \times n$ matrices over $N$, denoted by $M_n(N)$, is defined to be the subnearring of $M(N^n)$ generated by the set of maps $\left\{\, f^{a}_{ij} \colon N^n \rightarrow N^n ~\big|~ a \in N,~ 1\leq i,j \leq n \,\right\}$, where $f^a(x) = ax$, for all $a, x \in N$, $f^{a}_{ij}(x_1, \ldots , x_n) = (0, \ldots, \underbrace{ax_j}_{i\textsuperscript{th}}, \ldots, 0)$. Clearly, $f^{a}_{ij} = \iota_i f^a \pi_j$. If $N$ happens to be a ring, then $f^{a}_{ij}$ corresponds to the $n \times n$ matrix with $a$ in position $(i, j)$ and zeros elsewhere. We refer to \cite{meldrum1986matrix, meldrum1996intermediate, veldsman1992special, meyer1994chains} for further definitions and notations in matrix nearrings.
	
	For completeness we provide some of the results and definitions that have been frequently used in the article.
	
	The set of ideals of a nearring $N$ forms a complete modular lattice under inclusion, denoted by $\mathfrak L(N)$. The join and meet operations of this lattice are $+$ and $\cap$, respectively. There exist three maps $(~)^*$, $(~)^+$, and $(~)_*$ between $\mathfrak L(N)$ and $\mathfrak L(M_n(N))$, as defined below. The maps $(~)^*$ and $(~)^+$ are order-preserving injections from $\mathfrak L(N)$ to $\mathfrak L(M_n(N))$, and the map $(~)_*$ is an order-preserving surjection from $\mathfrak L(M_n(N))$ to $\mathfrak L(N)$.
	
	\begin{enumerate}[label = (\roman*)]
		\item $I^* = \left\{\, A \in M_n(N) \mid A\rho \in I^n~ \text{for all} ~\rho \in N^n \,\right\}$ is an ideal in $M_n(N)$ for every ideal $I$ of $N$.
		\item $I^+ = \left\langle\, f_{ij}^x \mid x \in I \,\right\rangle$ is an ideal in $M_n(N)$ contained in $I^*$, for every ideal $I$ of $N$.
		\item $\mathcal{I}_* = \left\{\, x \in N \mid x \in \im(\pi_j(A)) ~\text{for some}~ A \in \mathcal{I} \,\right\}$ is an ideal in $N$ for every ideal $\mathcal{I}$ of $M_n(N)$.
	\end{enumerate}
	
	\begin{theorem}[\cite{meldrum1986matrix}]\label{thm:lowerstar}
		Let $\mathcal{J}$ be an ideal of $M_n(N)$. Then 
		\begin{enumerate}[label = (\roman*)]
			\item $a \in \mathcal{J}_*$ if and only if $f_{ij}^a \in \mathcal{J}$ for $1 \leq i,j \leq n$.
			\item $\mathcal{J} \subseteq (\mathcal{J}_*)^*$.
		\end{enumerate}
	\end{theorem}
	
	\begin{proposition}[\cite{meldrum1996intermediate}]
		For any ideal $K$ in $N$, we have $(K^*)_* = K = (K^+)_*$.
	\end{proposition}
	
	\begin{lemma}[\cite{salvankar2024essential}]\label{lem:sumupperstar}
		Let $I,J\unlhd N$ be such that $I+J = N$. Then $I^*+J^* = M_n(N)$. 
	\end{lemma}
	
	Now, we recall the definition and basic properties of Galois connections \cite{GierzEtAl2003}, and then show that the maps defined above form Galois connections between the lattice of ideals of $N$ and the lattice of ideals of $M_n(N)$.
	
	\begin{definition}
		Let $(P, \leq)$ and $(Q, \leq)$ be partially ordered sets and $f \colon P \to Q$ and $g \colon Q \to P$ be monotone maps such that $f(p) \leq q$ if and only if $p \leq g(q)$, for all $p \in P$, $q \in Q$. Then $(f, g)$ is called a \emph{Galois connection} or an \emph{adjunction} between $P$ and $Q$. We call $f$ a \emph{left adjoint} of $g$ and $g$ a \emph{right adjoint} of $f$, denoted by $f \dashv g$.	
	\end{definition}
	
	\begin{proposition}
		Any Galois connection $(f, g)$ satisfies the following properties:
		\begin{enumerate}[label = (\roman*)]
			\item For monotone maps $f \colon P \to Q$ and $g \colon Q \to P$, the pair $(f, g)$ is a Galois connection if and only if $f \circ g \leq \id_Q$ and $\id_P \leq g \circ f$.
			\item If $f \dashv g$, then $f \circ g$ and $g \circ f$ are idempotent.
			\item If $f \dashv g$, then $f \circ g \circ f = f$ and $g \circ f \circ g = g$.
			\item Left adjoints preserve joins and right adjoints preserve meets.
		\end{enumerate}
	\end{proposition}

	\begin{theorem}
		The pairs of functions $((~)_*, (~)^*)$ and $((~)^+, (~)_*)$ are Galois connections between $\mathfrak L(N)$ and $\mathfrak L(M_n(N))$. 
	\end{theorem}
	
	\begin{proof}
		First, we show that $((~)_*, (~)^*)$ is a Galois connection. Let $\mathcal{I}$ be an ideal in $M_n(N)$ and $J$ be an ideal in $N$. Then $\mathcal{I}_* \subseteq J$ if and only if $\mathcal{I} \subseteq (\mathcal{I}_*)^* \subseteq J^*$ if and only if $\mathcal{I} \subseteq J^*$. Since $(~)_*$ and $(~)^*$ are order preserving, $((~)_*, (~)^*)$ is a Galois connection.
		
		Next, we show that $((~)^+, (~)_*)$ is a Galois connection. Let $I$ be an ideal in $N$ and $\mathcal{J}$ be an ideal in $M_n(N)$. Then $I^+ \subseteq \mathcal{J}$ if and only if for all $x\in I$, $f_{11}^x \in \mathcal{J}$ if and only if $x \in \mathcal{J}_*$. This implies that $I^+ \subseteq \mathcal{J}$ if and only if $I \subseteq \mathcal{J}_*$. Since $(~)_*$ and $(~)^+$ are order preserving, $((~)^+, (~)_*)$ is a Galois connection. 	
	\end{proof}
	
	The following properties are proved in \cite{meldrum1986matrix, meldrum1996intermediate, veldsman1992special, meyer1994chains}. But once we observe that $((~)_*, (~)^*)$ and $((~)^+, (~)_*)$ are Galois connections, these are direct consequences.
	\begin{proposition}\label{prop:GCProps}
		Let $I$ and $\mathcal I$ be ideals in $N$ and $M_n(N)$, respectively. Let $\{I_\alpha\}$ and $\{ \mathcal I_\alpha \}$ be collections of ideals in $N$ and $M_n(N)$ respectively. Then the following hold:
		\begin{enumerate}[leftmargin = *, label = (GC\arabic*)]
			\item \label{it:GC1} $(I^*)_* \subseteq I \subseteq (I^+)_*$.
			\item \label{it:GC2} $(\mathcal{I}_*)^+ \subseteq \mathcal{I} \subseteq (\mathcal{I}_*)^*$.
			\item \label{it:GC3} $((I^*)_*)^* = I^*$.
			\item \label{it:GC4} $((\mathcal{I}_*)^*)_* = \mathcal{I}_* = ((\mathcal{I}_*)^+)_*$.
			\item \label{it:GC5} $((I^+)_*)^+ = I^+$.
			\item \label{it:GC6} $(\sum\limits_\alpha \mathcal I_\alpha)_* = \sum\limits_\alpha (\mathcal I_\alpha)_*$.
			\item \label{it:GC7} $(\sum\limits_\alpha I_\alpha)^+ = \sum\limits_\alpha (I_\alpha)^+$.
			\item \label{it:GC8} $(\bigcap\limits_\alpha \mathcal I_\alpha)_* = \bigcap\limits_\alpha (\mathcal I_\alpha)_*$.
			\item \label{it:GC9} $(\bigcap\limits_\alpha I_\alpha)^* = \bigcap\limits_\alpha (I_\alpha)^*$.
		\end{enumerate}
	\end{proposition}
	
	\begin{lemma}\label{lem:K contained in J_*}
		Let $K\unlhd N$, $\mathcal{J}\unlhd M_n(N)$. If $K^* \subsetneq \mathcal{J}$ then $K\subsetneq \mathcal{J}_*$.
	\end{lemma}
	
	\begin{proof}
		Suppose that $K^*\subseteq \mathcal{J}$. Then since $((~)_*, (~)^*)$ is a Galois connection, $K \subseteq \mathcal{J}_*$. Let $K^* \neq \mathcal{J}$. On the contrary, suppose that $\mathcal{J}_* = K$. Since $((~)_*, (~)^*)$ is a Galois connection, $\mathcal{J}_* = K$ implies $\mathcal{J} \subseteq K^*$. Therefore $K^* = \mathcal{J}$.
	\end{proof}
	
	\begin{theorem}\label{thm:complement in M}
		If $K$ is a complement of $H$ in $N$, then $K^*$ is a complement of $H^*$ in $M_n(N)$.
	\end{theorem}
	
	\begin{proof}
		Since $K\cap H = \{0\}$ by \ref{it:GC9}, we get $K^*\cap H^* = \{0\}$. Suppose that there exists an ideal $\mathcal{I}$ of $M_n(N)$ such that $K^* \subsetneq \mathcal{I}$ and $\mathcal{I} \cap H^* = \{0\}$. By Lemma~\ref{lem:K contained in J_*}, $K\subsetneq \mathcal{I}_*$ and $\mathcal{I}_* \cap H = \{0\}$ which is a contradiction to $K$ is a complement of $H$. Hence $K^*$ is a complement of $H^*$. 
	\end{proof}

	\begin{theorem}\label{thm:complement in N}
		If $\mathcal{K}$ is a complement of $\mathcal{H}$ in $M_n(N)$, then $\mathcal{K}_*$ is a complement of $\mathcal{H}_*$ in $N$.
	\end{theorem}
	
	\begin{proof}
		Since $\mathcal{K}\cap \mathcal{H} = \{0\}$, by \ref{it:GC8} $\mathcal{K}_*\cap \mathcal{H}_* = \{0\}$. Suppose that there exists $I$ and ideal of $N$ such that $\mathcal{K}_* \subsetneq I$ and $I \cap \mathcal{H}_* = \{0\}$. Now by \ref{it:GC2}, $\mathcal{K}\subseteq(\mathcal{K}_*)^*\subsetneq I^*$ and $I^* \cap \mathcal{H} \subseteq I^* \cap (\mathcal{H}_*)^* = \{0\}$. Then by maximality of $\mathcal{K}$, we get $I^* = \mathcal{K}$. This implies that $I = (I^*)_* = \mathcal{K}_*$. Therefore $\mathcal{K}_*$ is a complement of $\mathcal{H}_*$ in $N$.
	\end{proof}
	
	By the fact that $I^+ \subseteq I^*$, one can observe that if $K$ is a complement of $H$ in $N$, then $K^+$ is not in general a complement of $H^+$ and $H^*$ in $M_n(N)$. In fact we can see in Theorem~\ref{thm:K=K_*^*} that if $K^+$ is a complement, then $K^* = K^+$, and also that the only complement ideals of $M_n(N)$ are full ideals.
	
	\begin{theorem}\label{thm:K=K_*^*}
		If $\mathcal{K}$ is a complement in $M_n(N)$, then $\mathcal{K} = (\mathcal{K}_*)^*$.
	\end{theorem}
	
	\begin{proof}
		Since $\mathcal{K}$ is a complement in $M_n(N)$, there exists an ideal $\mathcal{J}$ of $M_n(N)$ such that $\mathcal{K}\cap\mathcal{J}=\{0\}$ and any ideal $\mathcal{K}'$ properly containing $\mathcal{K}$ has non-zero intersection with $\mathcal{J}$. Since $\mathcal{K}\cap\mathcal{J}=\{0\}$ by \ref{it:GC8}, we have $\mathcal{K}_*\cap\mathcal{J}_*=\{0\}$. Therefore $(\mathcal{K}_*\cap\mathcal{J}_*)^*=\{0\}$. By \ref{it:GC9}, $(\mathcal{K}_*)^*\cap(\mathcal{J}_*)^*=\{0\}$. Since $\mathcal{J}\subseteq(\mathcal{J}_*)^*$, $(\mathcal{K}_*)^*\cap\mathcal{J}=\{0\}$. Since $\mathcal{K}\subseteq(\mathcal{K}_*)^*$ and $\mathcal{K}$ is a complement of $\mathcal{J}$, we have $\mathcal{K} = (\mathcal{K}_*)^*$.
	\end{proof}
	
	\begin{theorem}
		The mapping $P\rightarrow P^*$ defines a one-to-one correspondence between the sets of complement ideals of $N$ and $M_n(N)$.
	\end{theorem}
	
	\begin{proof}
		Follows from Theorems~\ref{thm:complement in M}, \ref{thm:complement in N} and \ref{thm:K=K_*^*}.
	\end{proof}
	
	\begin{lemma}
		Let $\mathcal{K}$ and $\mathcal{H}$ be ideals of $M_n(N)$. If $\mathcal{K}+\mathcal{H} = M_n(N)$, then $\mathcal{K}_*+ \mathcal{H}_* = N$.
	\end{lemma}
	
	\begin{proof}
		This follows from \ref{it:GC6} and the fact that $(M_n(N))_* = N$.
	\end{proof}
	
	\begin{theorem}\label{thm:supplementinN}
		If $\mathcal{K}$ is a supplement of $\mathcal{H}$ in $M_n(N)$, then $\mathcal{K}_*$ is a supplement of $\mathcal{H}_*$ in $N$.
	\end{theorem}
	
	\begin{proof}
		Since $\mathcal{K} + \mathcal{H} = M_n(N)$, by \ref{it:GC6} $\mathcal{K}_* + \mathcal{H}_* = N$. Suppose that there exists an ideal $I$ in $N$ such that $I \subseteq \mathcal{K}_*$ and $I + \mathcal{H}_* = N$. Since $(~)_*$ is a surjection, $I = \mathcal{J}_*$ for some $\mathcal{J}$ in $M_n(N)$. Thus, $\mathcal{J}_* + \mathcal{H}_* = N$, and by \ref{it:GC7} and \ref{it:GC2}, $(\mathcal{J}_*)^+ + \mathcal{H} = M_n(N)$. But, by \ref{it:GC2}, $(\mathcal{J}_*)^+ \subseteq (\mathcal{K}_*)^+ \subseteq \mathcal{K}$. Then, by minimality of $\mathcal{K}$, $(\mathcal{J}_*)^+ = \mathcal{K}$. Therefore by \ref{it:GC4}, $I = \mathcal{K}_*$. Hence, $\mathcal{K}_*$ is a supplement of $\mathcal{H}_*$ in $N$.
	\end{proof}
	
	\begin{theorem}\label{thm:I^+supplement in M}
		Let $K$ be a supplement ideal of $H$ in $N$. Then $K^+$ is a supplement ideal of $H^+$ in $M_n(N)$.
	\end{theorem}
	
	\begin{proof}
		Since $K + H =N$, $K^+ + H^+ = M_n(N)$. Suppose that $M_n(N)$ has an ideal $\mathcal{I} \subseteq K^+$ such that $\mathcal{I} + H^+ = M_n(N)$. Then, $\mathcal{I}_* + (H^+)_* = N$. Now since $(H^+)_* = H$, $\mathcal{I}_* + H = N$ and $\mathcal{I}_* \subseteq (K^+)_* = K$. Then by the minimality of $K$, $\mathcal{I}_* = K$, and hence $\mathcal I = K^+$. Therefore, $K^+$ is a supplement of $H^+$.
	\end{proof}
	
	By the fact that $I^+ \subseteq I^*$, one can observe that if $K$ is a supplement of $H$ in $N$, then $K^*$ is not in general a supplement of $H^+$ and $H^*$ in $M_n(N)$. In fact we can see in Theorem~\ref{thm:K=K_*^+} that if $K^*$ is a supplement, then $K^* = K^+$.
	
	\begin{theorem}\label{thm:K=K_*^+}
		If $\mathcal{K}$ is a supplement in $M_n(N)$, then $\mathcal{K} = (\mathcal{K}_*)^+$.
	\end{theorem}
	
	\begin{proof}
		Since $\mathcal{K}$ is a supplement in $M_n(N)$, there exists an ideal $\mathcal{H}$ in $M_n(N)$ such that $\mathcal{K} + \mathcal{H} = M_n(N)$ and for any ideal $\mathcal{K}' \subseteq \mathcal{K}$, $\mathcal{K'} + \mathcal{H} \neq M_n(N)$. Since $\mathcal{K} + \mathcal{H} = M_n(N)$ by \ref{it:GC6}, we have $\mathcal{K}_* + \mathcal{H}_* = N$. Therefore $(\mathcal{K}_*)^+ + (\mathcal{H}_*)^+ = M_n(N)$. Now by \ref{it:GC2}, $(\mathcal{K}_*)^+ + \mathcal{H} = M_n(N)$. Since $\mathcal{K}$ is a supplement, $\mathcal{K} \subseteq (\mathcal{K}_*)^+$. Therefore $\mathcal{K} = (\mathcal{K}_*)^+$.
	\end{proof}
	
	\begin{theorem}
		The mapping $P\rightarrow P^+$ defines a one-to-one correspondence between the sets of supplement ideals of $N$ and of $M_n(N)$.
	\end{theorem}
	
	\begin{proof}
		Follows from Theorems~\ref{thm:I^+supplement in M}, \ref{thm:supplementinN} and \ref{thm:K=K_*^+}.
	\end{proof}
	
	\begin{theorem}
		If $I$ is a direct summand in $N$, then $I^* = I^+$ in $M_n(N)$.
	\end{theorem}
	
	\begin{proof}
		Suppose that $I$ is a direct summand in $N$. Then there exists an ideal $J$ of $N$ such that $I + J = N$ and $I \cap J = \{0\}$. Now by Lemma~\ref{lem:sumupperstar}, $I^* + J^* = M_n(N)$, and by~\ref{it:GC9}, $I^* \cap J^* = \{0\}$. This implies that $I^*$ is a direct summand of $M_n(N)$. Now by Lemma~\ref{lem:DirectImpliesComplementAndSupplement}, $I^*$ is both supplement and complement in $M_n(N)$. Therefore by Theorem~\ref{thm:K=K_*^+}, $I^* = I^+$. 
	\end{proof}
	
	\begin{corollary}
		If $N = I \oplus J$, then $M_n(N) = I^* \oplus J^*$.
	\end{corollary}
	
	\section{Complement and Supplement Ideal Graphs}\label{sec:Graphs}
	
	In this section, we define the complement ideal graph of the nearring $N$, denoted by $\CIG(N)$ and the supplement ideal graph of $N$, denoted by $\SIG(N)$.  
	
	Throughout this section the graphs considered are simple and undirected. For vertices $u$ and $v$ of a graph, we write $u \sim v$ to mean that $u$ and $v$ are adjacent, and $u \nsim v$ to mean that they are non-adjacent. A vertex in $G$ is called a pendant vertex if it is adjacent to a unique vertex. The \emph{girth} of a graph $G$ is the length of a smallest cycle contained in  $G$, and is denoted by $\gr(G)$. A maximal complete subgraph of $G$ is a \emph{clique} of $G$, and a clique of largest order is a \emph{maximum clique}. The \emph{clique number} of $G$ is the order of a maximum clique of $G$, and is denoted by $\omega(G)$. We denote the complete graph on $n$ vertices as $K_n$. 
	
	\begin{definition}
		The \emph{complement ideal graph} of a nearring $N$ is the graph whose vertices are ideals of $N$, with two distinct vertices $I$ and $J$ being adjacent if and only if $I + J$ is a complement ideal in $N$. It is denoted by $\CIG(N)$.
	\end{definition}
	
	\begin{definition}
		The \emph{supplement ideal graph} of a nearring $N$ is the graph whose vertices are ideals of $N$, with two distinct vertices $I$ and $J$ being adjacent if and only if $I \cap J$ is a supplement in $N$. It is denoted by $\SIG(N)$.
	\end{definition}
	
	In the remark given below, we make basic observations that easily follow from the definitions, but are important and will be used repeatedly in proving the main results. These observations also show that complement and supplement ideals of $N$ can be immediately identified from $\CIG(N)$ and $\SIG(N)$ respectively.
	
	\begin{remark}\label{rem:0andNinCIGSIG}
		We can observe the following:
		\begin{enumerate}[label = (\roman*)]
			\item\label{it:0andNinCIGSIG1} In $\CIG(N)$, the only ideals adjacent to $\{0\}$ are complement ideals of $N$, and $N$ is a universal vertex.
			\item\label{it:0andNinCIGSIG2} In $\SIG(N)$, the only ideals adjacent to $N$ are supplement ideals of $N$, and $\{0\}$ is a universal vertex.
		\end{enumerate}
	\end{remark}
	
	It is evident from Remark~\ref{rem:0andNinCIGSIG} \ref{it:0andNinCIGSIG2} that if $\SIG(N)$ is a complete graph, then every ideal is a supplement. Conversely, if every ideal is a supplement, then the intersection of any two ideals is a supplement, and hence $\SIG(N)$ is complete. Thus, we have the following result.
	
	\begin{corollary}\label{cor:complete}
		$\SIG(N)$ is a complete graph if and only if every non-trivial ideal of $N$ is a supplement.
	\end{corollary}
	
	\begin{lemma}\label{lem:minimalisuniversal}
		Any minimal ideal of $N$ is a universal vertex in $\SIG(N)$ if and only if it is a supplement.
	\end{lemma}
	
	\begin{proof}
		Let $I$ be a minimal ideal which is a supplement and $J$ be any ideal in $N$. Now we have two cases:
		
		\noindent	Case 1: $I \subseteq J$. Then $I \cap J = I$. Therefore $I \sim J$.
		
		\noindent	Case 2: $I \nsubseteq J$. Then, since $I$ is minimal, $I \cap J = \{0\}$, which is a supplement. Therefore $I \sim J$.
		
		\noindent	Thus $I$ is a universal vertex in $\SIG(N)$. The converse is straightforward.
	\end{proof}

	\begin{remark}
		In Lemma~\ref{lem:minimalisuniversal}, the condition of being minimal is not necessary for an ideal to be a universal vertex in $\SIG(N)$. 	
	\end{remark}
	
	\begin{example}
		Let $N = N_1 \times N_2$, where $N_1 = (\mathbb Z_3, +, \star)$ and $N_2 = (\mathbb Z_{12}, +, \star)$ (\texttt{LibraryNearRing(3/1, 1)} and \texttt{LibraryNearRing(12/1, 1)} in \cite{SONATA}). The ideal lattice and the graph $\SIG(N)$ are given in Figure~\ref{fig:Z3xZ12}.
		\begin{figure}[!h]
			\centering
			\subfloat[$\mathfrak L(N)$]{
				\includegraphics[width=0.33\textwidth]{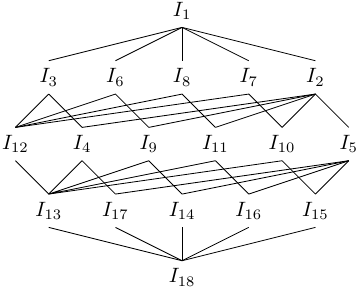}
			}
			\hspace{0.3\textwidth}
			\subfloat[$\SIG(N)$]{
				\includegraphics[width=0.25\textwidth]{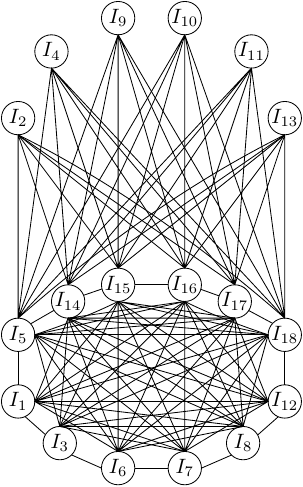}
			}
			\caption{}\label{fig:Z3xZ12}
		\end{figure}		
		Observe that $I_5$ is a universal vertex of $\SIG(N)$, but is not a minimal ideal of $N$.
	\end{example}
	
	\begin{remark}\label{rem:nbdofm}
		If $\mathfrak{m}$ is a minimal ideal which is not a supplement in $N$, then the neighborhood of $\mathfrak{m}$ in $\SIG(N)$ consists of the ideals of $N$ not containing $\mathfrak{m}$. 
	\end{remark}
	
	\begin{lemma}\label{lem:uniqueuniversal}
		Let $N$ be a nearring with DCCI. Then no minimal ideal of $N$ is a supplement ideal if and only if $\SIG(N)$ has a unique universal vertex.
	\end{lemma}
	
	\begin{proof}
		Let $I$ be any non-trivial ideal of $N$. Since $N$ has DCCI, there exists a minimal ideal $\mathfrak{m}$ such that $\mathfrak{m}\subseteq I$. If $I = \mathfrak m$, then by Lemma~\ref{lem:minimalisuniversal}, it is not universal. Otherwise, by Remark~\ref{rem:nbdofm}, $I \nsim \mathfrak m$. The converse follows from Lemma~\ref{lem:minimalisuniversal}.
	\end{proof}
	
	\begin{corollary}
		Let $N$ be an indecomposable nearring with DCCI. Then $\SIG(N)$ has a unique universal vertex.
	\end{corollary}
	
	\begin{proof}
		By Lemma~\ref{lem:indecomposable} no minimal ideal of $N$ is a supplement. Thus the result follows from Lemma~\ref{lem:uniqueuniversal}.
	\end{proof}
	
	The \emph{socle} of $N$, denoted by $\operatorname{soc}(N)$, is the sum of all minimal ideals of $N$, and the \emph{radical} of $N$, denoted by $\operatorname{Rad}(N)$, is the intersection of all maximal ideals of $N$.
	
	\begin{theorem}\label{thm:pendantCIG}
		If an ideal $I$ is a pendant vertex in $\CIG(N)$, then $\operatorname{soc}(N) \subseteq I \subseteq \operatorname{Rad}(N)$. 
	\end{theorem}
	\begin{proof}
		Suppose that $I$ is a pendant vertex. Let $M$ be a maximal ideal of $N$. Since $I$ is a pendant vertex, $I + M \neq N$. Therefore $I \subseteq M$. Hence every maximal ideal contains $I$. Suppose that $\mathfrak{m}$ is a minimal ideal of $N$ such that $\mathfrak{m}\nsubseteq I$. Then $I \cap \mathfrak{m} = \{0\}$. If $I$ is not maximal with respect to this property, then consider a maximal chain $\{I_\alpha\}$ of ideals containing $I$ and $I_\alpha\cap \mathfrak{m} = \{0\}$. Then there exists $K = \bigcup I_\alpha$ such that $K \cap \mathfrak{m} = \{0\}$. That implies $K$ is a complement ideal. Thus we get $I$ is adjacent to $K$, which is a contradiction. Therefore, $\mathfrak{m}\subseteq I$. Hence every minimal ideal is contained in $I$.    
	\end{proof}
	
	As arbitrary intersection is not distributive over sum, in case of $\SIG(N)$, Theorem~\ref{thm:pendantCIG} will be true when we consider $N$ with DCCI. The statement changes as follows:
	
	\begin{theorem}\label{thm:pendantSIG}
		If an ideal $I$ is a pendant vertex in $\SIG(N)$ and $N$ has DCCI, then $\operatorname{soc}(N) \subseteq I \subseteq \operatorname{Rad}(N)$. 
	\end{theorem}
	
	If $N$ has ACCI and DCCI, then the converses of Theorems~\ref{thm:pendantSIG} and \ref{thm:pendantCIG} can also be obtained.
	
	\begin{theorem}\label{thm:pendantvertex}
		Let $N$ be a nearring with ACCI and DCCI. Then a proper non-trivial ideal $I$ of $N$ is a pendant vertex in $\SIG(N)$ ($\CIG(N)$) if and only if  $\operatorname{soc}(N) \subseteq I \subseteq \operatorname{Rad}(N)$.
	\end{theorem}
	\begin{proof}
		Let $I$ be a proper non-trivial ideal of $N$ such that $\mathfrak{m} \subseteq I \subseteq M$, for every maximal ideal $M$ and every minimal ideal $\mathfrak{m}$. Since $I$ is contained in every maximal ideal, no non-trivial ideal contained in $I$ can be a supplement. 
		Thus, if $J$ is any non-trivial ideal of $N$, then $I \cap J$ is a non-trivial ideal of $N$ contained in $I$, and hence, is not a supplement. Thus $I$ is not adjacent to $J$. Therefore, $I$ is a pendant vertex in $\SIG(N)$. By duality, the same argument holds for $\CIG(N)$.
	\end{proof}
	
	\begin{proposition}
		If the ideals of $N$ form a chain, then $\SIG(N)$ is a star graph.
	\end{proposition}
	\begin{proof}
		Since ideals of $N$ form a chain, no proper non-trivial ideal is a supplement. Therefore, $\SIG(N)$ is a star graph. 
	\end{proof}
	
	\begin{remark}
		In particular, for nearrings defined on the following groups, $\SIG(N)$ is a star graph.
		\begin{enumerate}[label=(\roman*)]
			\item The cyclic $p$-group $(\mathbb{Z}_{p^n}, +)$
			\item The semidirect product $\mathbb{Z}_p\rtimes H$, where $H$ is a finite simple group whose order is relatively prime to $p$
			\item The symmetric group $S_n$
		\end{enumerate}
		where $p$ is a prime and $n$ is any positive integer.
	\end{remark}	
	
	\begin{theorem}
		Let $|N| = pq$, where $p$ and $q$ are distinct primes. Then $\SIG(N)$ is either $K_4$ or $P_3$ or $P_2$.
	\end{theorem}
	
	\begin{proof}
		The possible ideals of $N$ are $\{0\}$, $N$, a subgroup generated by an element of order $p$, say $\langle a \rangle$, and a subgroup generated by an element of order $q$, say $\langle b\rangle$. If $\{0\}$, $N$, $\langle a \rangle$ and $\langle b \rangle$ are the ideals of $N$, then every non-trivial ideal is a supplement in $N$. Therefore, by Corollary~\ref{cor:complete}, we have $\SIG(N) \cong K_4$. If the ideals are $\{0\}$, $N$, and exactly one of $\langle a \rangle$ and $\langle b \rangle$, then $\SIG(N) \cong P_3$. If the ideals are only $\{0\}$ and $N$, then $\SIG(N) \cong P_2$.
	\end{proof}
	
	\begin{lemma}\label{lem:SupplementAdjacency}
		Let $I$ be an ideal of $N$, and $J$ a supplement ideal of $I$. Then $I$ is adjacent to $J$ in $\SIG(N)$ if and only if $I \cap J = \{0\}$.
	\end{lemma}
	\begin{proof}
		Suppose that $J$ is a supplement ideal of $I$, and $I$ is adjacent to $J$ in $\SIG(N)$. Then by Lemma~\ref{lem:IcapJissuperfluous}, $I \cap J$ must be a superfluous ideal. Thus, by Remark~\ref{rem:SuperfluousSupplement}, $I \cap J = \{0\}$. The converse is obvious.
	\end{proof}
	
	\begin{theorem}\label{thm:uniquemaximalandminimal}
		Let $N$ have ACCI and DCCI. Then $N$ has a unique maximal ideal and a unique minimal ideal if and only if $\SIG(N)$ is a star graph.
	\end{theorem}
	
	\begin{proof}
		Suppose that $\SIG(N)$ is a star graph. Then for any non-trivial proper ideal $I$  of $N$, $I\cap N = I$ is not a supplement. Thus $\{0\}$ and $N$ are the only supplements in $N$. Now, by Lemma~\ref{lem:UniqueMaximalIdeal}, $N$ has a unique maximal ideal. If $N$ has two minimal ideals say, $\mathfrak{m_1}$ and $\mathfrak{m_2}$, then $\mathfrak{m_1}$ is adjacent to $\mathfrak{m_2}$, which is a contradiction. Therefore, $N$ has unique minimal ideal. Conversely suppose that $N$ has unique maximal and minimal ideals. Then by Lemma~\ref{lem:UniqueMaximalIdeal}, $\{0\}$ and $N$ are the only supplements in $N$. Since $\{0\}$ is a universal vertex, the star graph with $\{0\}$ as the central vertex and all non-trivial ideals as the pendant vertices is a subgraph of $\SIG(N)$. Suppose that $I$ and $J$ are two distinct non-trivial ideals of $N$ such that $I$ is adjacent to $J$. Then $I \cap J = K$ is a supplement in $N$. Since $\{0\}$ and $N$ are the only supplements, $K = \{0\}$. But since $N$ has DCCI and a unique minimal ideal, $K \neq \{0\}$, which is a contradiction. Therefore, $\SIG(N)$ is a star graph.
	\end{proof}
	
	\begin{lemma}\label{lem:star}
		For any nearring $N$, either $\SIG(N)$ is a star graph or $\gr(\SIG(N)) = 3$.
	\end{lemma}
	
	\begin{proof}
		Suppose that $\SIG(N)$ is not a star graph. Since $\{0\}$ is a universal vertex of $\SIG(N)$, this implies that there exist non-trivial ideals $I$ and $J$ such that $ I\sim J$. Therefore $\gr(\SIG(N)) = 3$.
	\end{proof}
	
	\begin{lemma}
		Let $N$ have DCCI and contain a unique minimal ideal. Then, $\{0\}$ and $N$ are the only supplements in $N$ if and only if $\SIG(N)$ is triangle free.
	\end{lemma}
	
	\begin{proof}
		Suppose that $\{0\}$ and $N$ are the only supplements in $N$ and $\SIG(N)$ has a triangle. Then there exist ideals $I$, $J$ and $K$ such that $I \sim J \sim K \sim I$ in $\SIG(N)$. Since $N$ has DCCI and contains a unique minimal ideal, no two non-trivial ideals can have trivial intersection. Therefore one of the ideals should be $\{0\}$ and one of the ideals should be $N$. Let $J = \{0\}$ and $K = N$. Since $I \sim N$, this implies $I$ is a supplement, which is a contradiction. Therefore $\SIG(N)$ is triangle free. The converse follows from Lemma~\ref{lem:star}. 
	\end{proof}
	
	\begin{lemma}\label{lem:K4issubgraph}
		Let $N$ have ACCI. If $\SIG(N)$ has $K_4$, then $N$ has a proper non-trivial ideal that is either a supplement or a complement.
	\end{lemma}
	
	\begin{proof}
		Suppose that $\SIG(N)$ has $K_4$. Then we have two cases. 
		
		Case 1: $N$ is a vertex in a $K_4$. Then there exist proper non-trivial ideals $I$ and $J$ such that $I \sim N$ and $J \sim N$. This implies that $I$ and $J$ are supplements.
		
		Case 2: $N$ is not a vertex in any $K_4$ in $\SIG(N)$. Then there exist proper non-trivial ideals $I$, $J$ and $K$ such that $\{0\}$, $I$, $J$ and $K$ form $K_4$. Since $N$ is not in any $K_4$, no proper non-trivial ideal is a supplement. Also, since $I \sim J$, $I\cap J = \{0\}$. Since $N$ has ACCI, there exists a proper non-trivial ideal $L$ such that $I \subseteq L$ and is maximal with respect to the property $L \cap J =\{0\}$. Hence $L$ is a complement ideal of $J$.
	\end{proof}
	
	\begin{lemma}
		If $N$ has a proper non-trivial supplement ideal, then $\omega(\SIG(N)) \ge 3$.  
	\end{lemma}
	
	The \emph{tensor product} of graphs $G$ and $H$ is the graph $G \otimes H$ with vertex set $V(G) \times V(H)$ in which any two vertices $(u_1, v_1)$ and $(u_2, v_2)$ are adjacent if and only if $u_1 \sim u_2$ in $G$ and $v_1 \sim v_2$ in $H$.
	
	For the next result we consider the graph $\SIG(N)$ with loops, denoted by $\SIG^\circ(N)$, which is the same as $\SIG(N)$, but with a loop on precisely each supplement ideal of $N$.

	\begin{theorem}
		If $N$ is a nearring with right identity, and $N = N_1 \oplus N_2$ is a direct decomposition of $N$, then $\SIG^\circ(N) = \SIG^\circ(N_1) \otimes \SIG^\circ(N_2)$.
	\end{theorem}
	
	\begin{proof}
		Let $I = I_1 \oplus I_2$ and $J = J_1 \oplus J_2$ be any two ideals of $N$. Then $I \sim J$ in $\SIG^\circ(N)$ if and only if $I \cap J$ is a supplement in $N$. From Lemma~\ref{lem:supplementsindirectsum}, we know that $I \cap J = (I_1 \cap J_1) \oplus (I_2 \cap J_2)$ is a supplement if and only if $I_\alpha \cap J_\alpha$ is a supplement in $N_\alpha$, for $\alpha = 1, 2$. Equivalently, $I_\alpha \sim J_\alpha$ in $\SIG^\circ(N_\alpha)$, for $\alpha = 1, 2$. It follows that $\SIG^\circ(N) = \SIG^\circ(N_1) \otimes \SIG^\circ(N_2)$.
	\end{proof}
	
	\begin{theorem}
		The following are equivalent:
		\begin{enumerate}[label=(\roman*)]
			\item \label{it:CompleteGraph1} Every ideal of $N$ is a direct summand.
			\item \label{it:CompleteGraph2} $\SIG(N)$ is a complete graph.
			\item \label{it:CompleteGraph3} $\CIG(N)$ is a complete graph.
		\end{enumerate}
	\end{theorem}
	
	\begin{proof}
		The proofs of \ref{it:CompleteGraph1} $\implies$ \ref{it:CompleteGraph2} and \ref{it:CompleteGraph1} $\implies$ \ref{it:CompleteGraph3} are straightforward.
		
		\ref{it:CompleteGraph2} $\implies$ \ref{it:CompleteGraph1}: If $\SIG(N)$ is complete, then every ideal is a supplement. Thus, by Lemma~\ref{lem:SupplementAdjacency}, every ideal is a direct summand.
		
		\ref{it:CompleteGraph3} $\implies$ \ref{it:CompleteGraph1} follows by duality.
	\end{proof}
	
	A map $f \colon V(G_1) \to V(G_2)$ where $G_1$ and $G_2$ are graphs is a \emph{weak graph homomorphism} if whenever $x \sim y$ in $G_1$, then either $f(x) = f(y)$ or $f(x) \sim f(y)$ in $G_2$ (\cite{knauer2019algebraic}).
	
	\begin{theorem}
		For a nearring $N$ with identity, the map $(~)_*$ is a weak graph homomorphism from $\CIG(M_n(N))$ to $\CIG(N)$. 
	\end{theorem}
	
	\begin{proof}
		Suppose that $\mathcal{I} \sim \mathcal{J}$ in $\CIG(M_n(N))$. Then $\mathcal{I} + \mathcal{J}$ is a complement. Now we have either $\mathcal{I}_* = \mathcal{J}_*$ or by Theorem~\ref{thm:complement in N}, $\mathcal{I}_* + \mathcal{J}_*$ is a complement. Therefore $(~)_*$ is a weak graph homomorphism.
	\end{proof}
	
	\begin{theorem}\label{thm:maloneuniversal}
		Let $N$ be a Malone trivial nearring. Then $\CIG(N)$ and $\SIG(N)$ each have exactly one universal vertex.
	\end{theorem}
	\begin{proof}
		Let $I$ be any proper ideal of $N$. Since $N$ is indecomposable, by Lemma~\ref{lem:indecomposable}, the unique maximal ideal $M$ of $N$ cannot be a complement. Now, by Lemma~\ref{lem:maximalideal}, $I\subseteq M$. Therefore, by duality applied to Remark~\ref{rem:nbdofm}, $I$ is not adjacent to $M$ in $\CIG(N)$. Now by Lemma~\ref{lem:N!=I+J}, $I$ is not adjacent to $N$ in $\SIG(N)$. Therefore, $\CIG(N)$ and $\SIG(N)$ each have exactly one universal vertex. 
	\end{proof}
	
	We conclude this section by applying the theory developed here to planar and Malone trivial nearrings in general, and in particular to the case where the underlying group is a finite elementary Abelian $2$-group.
	
	\begin{corollary}\label{cor:maloneclique}
		Let $N$ be a Malone trivial nearring. If $\omega(\SIG(N)) = 2$, then $N$ has a unique minimal ideal. 
	\end{corollary}
	
	\begin{proof}
		The proof follows from Lemma~\ref{lem:N!=I+J} and Theorem~\ref{thm:uniquemaximalandminimal}.
	\end{proof}
	
	\begin{lemma}[{\cite[Teorema 1]{cotti1986radicali}, \cite[Theorem 6.1]{ke2003recent}}]\label{lem:greatestideal}
		Let $N$ be a planar nearring. Then there exists a greatest proper ideal $D$ in $N$, which is the sum of all proper left ideals.
	\end{lemma}
	
	\begin{lemma}[{\cite[Theorem 6.2]{ke2003recent}}]\label{lem:idealsofplanarnr}
		Let $N$ be a planar nearring and $D$ its greatest proper ideal. Then the proper left ideals of $N$ are precisely the additive normal subgroups of $N$ contained in $D$.
	\end{lemma}
	
	Similarly, Theorem~\ref{thm:maloneuniversal} and Corollary~\ref{cor:maloneclique} can be proved for planar nearrings using Lemmas~\ref{lem:greatestideal} and \ref{lem:idealsofplanarnr}.
	
	Now, we investigate the complement ideal graphs and supplement ideal graphs of planar and Malone trivial nearrings defined on elementary Abelian $2$-groups. Recall that every finite elementary Abelian $2$-group $(\mathbb{Z}_2^n, +)$ forms an $n$-dimensional vector space over $\mathbb{Z}_2$. In a vector space, every subspace is a direct summand. Consider $N$ to be a planar nearring or a Malone trivial nearring on $(\mathbb{Z}_2 ^n, +)$ and fix a subspace $M$ of dimension $m$ to be the unique maximal ideal. If $N$ is a planar nearring, then by Lemmas~\ref{lem:greatestideal} and \ref{lem:idealsofplanarnr}, the proper ideals of $N$ are exactly the subspaces of $M$. If $N$ is a Malone trivial nearring, then by the property of ideals in Malone trivial nearring, the proper ideals of $N$ are exactly the subspaces of $M$. Therefore we can associate dimension to ideals of $N$. Also, every ideal except $M$ is a complement ideal. Hence any two proper ideals will be adjacent in $\CIG(N)$ if and only if their sum is not $M$. In particular, if $m$ is even, then any two $\frac m 2$-dimensional ideals are adjacent in $\CIG(N)$ if and only if they have a non-trivial intersection.
	
	We use the following results from \cite{chajda2019lattice} to prove some properties of $\CIG(N)$ and $\SIG(N)$.
	
	Let $m$ be the dimension of a vector space $V$ over $F$ and $q$ be the cardinality of the field $F$. Define $a_0 = 1$, $a_n = \prod_{i=1}^{n} (q^i - 1)$. Then
	\begin{enumerate}[label = (\roman*)]
		\item the number of $d$-dimensional subspaces of $V$ is $\frac {a_m}{a_d a_{m-d}}$.
		\item the number of complements of a $d$-dimensional subspace is $q^{d(m-d)}$.
		\item the number of $e$-dimensional subspaces containing a $d$-dimensional subspace is $\frac {a_{m-d}}{a_{m-e} a_{e-d}}$.	
	\end{enumerate}
	
	\vspace{0.5cm}
	
	Throughout the rest of this section, let $N = (\mathbb{Z}_2 ^n, +, \cdot)$ be a planar or Malone trivial nearring and $M$ be its unique maximal ideal. In the following results, we determine the maximum cliques of $\CIG(N)$ and $\SIG(N)$, and also give expressions for the clique numbers of these graphs. In $\CIG(N)$, two cases occur depending on the dimension of $M$ -- if the dimension is odd, then there is a unique maximum clique, while if the dimension is even, then the maximum cliques are not unique. The number of maximum cliques in the latter case is also given.
	
	\begin{theorem}\label{lem:odddim}
		If $M$ is of odd dimension $m$, then the ideals with dimension at most $ \lfloor \frac m 2 \rfloor$ together with $N$ form a maximum clique in $\CIG(N)$.
	\end{theorem}
	
	\begin{proof}
		Since every ideal of $N$ except $M$ is a complement and sum of two ideals with dimension at most $ \lfloor \frac m 2 \rfloor$ is not $M$, we get a complete subgraph $H$, with vertex set containing the ideals with dimension at most $\lfloor \frac m 2 \rfloor$, which is $\sum_{i=0}^{ \lfloor \frac m 2 \rfloor} \left( \frac {a_m}{a_i a_{m-i}} \right)$ in number, together with $N$. Now, it remains to show that $H$ is a maximal complete subgraph. Consider a subgraph $K$ containing an $(m-1)$-dimensional ideal $I$. Now, $K$ can contain at most $\sum_{i=0}^{m-1} \left( \frac {a_{m-1}}{a_i a_{m-1-i}} \right)$ number of ideals of dimension less than $m-1$ and adjacent to $I$. This number is clearly less than the number of vertices in $H$. Therefore, $K$ cannot be a maximum clique. Hence the ideals with dimension at most $ \lfloor \frac m 2 \rfloor$ together with $N$ form a maximum clique in $\CIG(N)$.
	\end{proof}
	
	The clique number of $\CIG(N)$ when the dimension of $M$ is odd is given in the corollary below.
	
	\begin{corollary}
		If $M$ is of odd dimension $m$, then
		\begin{equation*}
			\omega(\CIG(N)) = \sum_{i=0}^{ \left\lfloor \frac m 2 \right\rfloor} \left( \frac {a_m}{a_i a_{m-i}} \right) + 1.
		\end{equation*}
	\end{corollary}
	
	\begin{theorem}\label{lem:evendim}
		If $M$ is of even dimension $m \geq 4$, then the ideals with dimension less than $\frac m 2$ together with $N$ and  $\frac {a_{m-1}} {a_ {\frac m 2}a_{\frac m 2 - 1}}$ number of ideals of dimension $\frac m 2$ form a maximum clique in $\CIG(N)$ and they are $2^{m+1} -2$ in number.
	\end{theorem}
	
	\begin{proof}
		Let $I$ and $J$ be two ideals whose dimensions are less than $\frac m 2$. Since $I + J \neq M$ and every ideal of $N$ except $M$ is a complement, $I \sim J$ in $\CIG(N)$. Hence, the ideals with dimension less than $\frac m 2$, together with $N$, form a complete subgraph $H$. Also, every ideal of dimension $\frac m 2$ is adjacent to every vertex of $H$.
		
		Suppose that $K$ and $L$ are two ideals of dimension $\frac m 2$. Then $K \sim L$ if and only if $K + L \neq M$. That is, $K + L \subseteq P$ where $P$ is an ideal of dimension $m-1$. Therefore we can choose $\frac {a_{m-1}} {a_ {\frac m 2}a_{\frac m 2 - 1}}$ number of ideals with dimension $\frac m 2$ which are contained in $P$ or contains a $1$-dimensional ideal and are mutually adjacent in $\CIG(N)$. As the dimension of ideals considered increases, the number of ideals contained in $P$ decreases. Therefore $H$, together with $\frac {a_{m-1}} {a_ {\frac m 2}a_{\frac m 2 - 1}}$ number of $\frac m 2$-dimensional ideals, form a maximum clique. Since there are two ways to choose vertices to form a clique, i.e., the number of ideals contained in an $(m-1)$-dimensional ideal and the number of ideals containing a $1$-dimensional ideal, and the number of $1$-dimensional ideals are $2^m-1$ in number, we get $2\times (2^m-1)$ number of maximum cliques. 
	\end{proof}
	
	The clique number of $\CIG(N)$ when the dimension of $M$ is even is given in the corollary below.
	
	\begin{corollary}
		If $M$ is of even dimension $m \geq 4$, then
		\begin{equation*}
			\omega(\CIG(N)) = \sum_{i=0}^{\frac m 2 -1} \left( \frac {a_m}{a_i a_{m-i}} \right) + \frac {a_{m-1}} {a_ {\frac m 2}a_{\frac m 2 - 1}} + 1.
		\end{equation*}
	\end{corollary}
	
	\begin{remark}
		If $N$ is a planar nearring, then no proper non-trivial ideal of $N$ is a supplement. Therefore any two ideals are adjacent in $\SIG(N)$ if and only if their intersection is trivial.
	\end{remark}
	
	\begin{theorem}
		Let $M$ be of dimension $m$. Then the ideals of dimension $1$ together with $\{0\}$ form a maximum clique in $\SIG(N)$ of size $2^m$.
	\end{theorem}
	
	\begin{proof}
		Let $I$ and $J$ be $1$-dimensional ideals. Since $I \cap J = \{0\}$, $I$ is adjacent to $J$ in $\SIG(N)$. Therefore the $1$-dimensional ideals together with $\{0\}$ form a complete subgraph $H$ of order $2^m -1 + 1 = 2^m$. Since any ideal with dimension greater than $1$ contains at least one $1$-dimensional ideal, it cannot be adjacent to every vertex in $H$. Therefore $H$ is a maximal complete subgraph. Suppose that some of the $1$-dimensional ideals together with higher dimensional ideals form a complete subgraph say, $K$. Then choose $r$ number of $1$-dimensional ideals which generate an $r$-dimensional ideal. That is, we can choose $2^r - 1$ number of $1$-dimensional ideals which are adjacent to $r$-dimensional ideals which are generated by remaining $m-r$ ideals of dimension $1$. Therefore we can select $\lfloor \frac {2^{m-r} - 1} {2^i - 1} \rfloor$ number of $i$-dimensional ideals which are to be vertices of $K$. Since ideals of dimension greater than $\lfloor \frac m 2 \rfloor$ cannot be adjacent to any ideals of dimension  greater than or equal to $\lfloor \frac m 2 \rfloor$, $i$ varies from $2$ to $\lfloor \frac m 2 \rfloor$. Therefore the maximum number of vertices in $K$ is 
		\begin{equation}\label{eq:cliqueno}
			\sum_{i= 2}^{\left\lfloor \frac m 2 \right\rfloor} \left\lfloor \frac {2^{m-r} - 1} {2^i - 1} \right\rfloor  + 2^r - 1 \leq 2^{m-r} - 1 + 2^r - 1, 
		\end{equation}
		which is maximum only when $r = m$ or $0$. If $r = 0$, then note that the upper bound given in \eqref{eq:cliqueno} cannot be attained. Therefore $H$ is the unique maximum clique.
	\end{proof}
	
	\begin{example}
		The graphs $\CIG(N)$ and $\SIG(N)$ of the Malone trivial nearring $N$ defined on $\mathbb Z_2^4$ are shown in Figure~\ref{fig:MTEA(2,3)}.
		\begin{figure}[h]
			\centering
			\subfloat[$\CIG(N)$]{
				\includegraphics[scale=0.58]{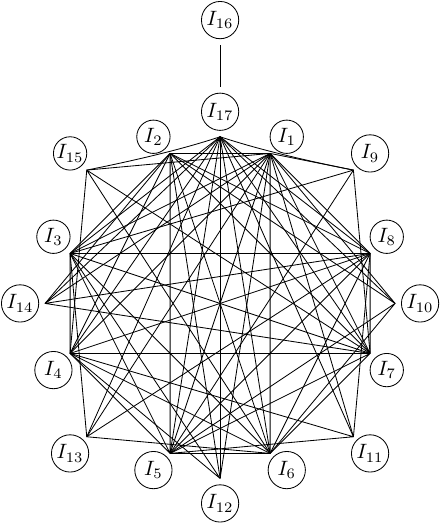}
				
			}
			\hspace{0.1\textwidth}
			\subfloat[$\SIG(N)$]{
				\includegraphics[scale=0.58]{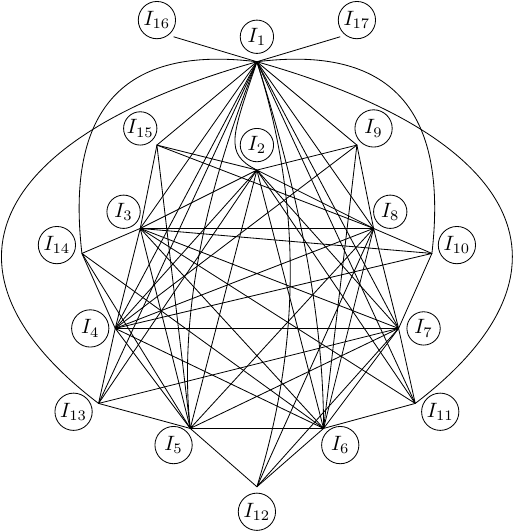}
			}
			\caption{Complement and supplement ideal graphs of Malone trivial nearring $N = \mathbb Z_2^4$}\label{fig:MTEA(2,3)}
		\end{figure}
		In $\CIG(N)$, the vertices $I_1$, $I_2$, $I_3$, $I_4$, $I_5$, $I_6$, $I_7$, $I_8$, and $I_{17}$ form the unique maximum clique, and in $\SIG(N)$, the vertices $I_1$, $I_2$, $I_3$, $I_4$, $I_5$, $I_6$, $I_7$, $I_8$ form the unique maximum clique.
	\end{example}
	
	\section*{Conclusion}
	In this article, we studied fundamental properties of supplement ideals and their dual concept, complement ideals, in nearrings. We observed that there exist Galois connections between $\mathfrak L(N)$ and $\mathfrak L(M_n(N))$ and gave one-to-one correspondences between complement and supplement ideals of nearrings and those of their matrix nearrings. Finally, we defined graphs associated with supplement and complement ideals of nearrings and investigated their structure. In particular, we obtained the clique numbers of these graphs for planar and Malone trivial nearrings defined on $(\mathbb Z_2^n, +)$.

	\section*{Acknowledgement} 
	The first author acknowledges Dr.~TMA Pai Fellowship, MAHE, Manipal for financial support. All the authors acknowledge Manipal Institute of Technology (Manipal), Manipal Academy of Higher Education for kind encouragement.

	\bibliographystyle{plainnat}
	\bibliography{bibliography}
\end{document}